\documentclass[final,3p,times]{elsarticle}
\biboptions{sort}

\usepackage{amsmath,amssymb,amsthm,mathtools,mathrsfs}
\usepackage{graphicx}
\usepackage{needspace}
\usepackage{courier}
\usepackage[hidelinks]{hyperref}
\hypersetup{pdftitle={The Morse Index, Nullity and Jacobi Fields of Constant-Curvature Minimal 2-Spheres in Complex Projective Spaces},
 pdfauthor={Hongbin Cui; Shuping Huang}}
\allowdisplaybreaks[1]

\newtheorem{thm}{Theorem}[section]
\newtheorem{lem}[thm]{Lemma}
\newtheorem{prop}[thm]{Proposition}
\newtheorem{cor}[thm]{Corollary}
\theoremstyle{remark}
\newtheorem*{remark}{Remark}
\theoremstyle{plain}

\newcommand{\R}{\mathbb{R}}
\newcommand{\C}{\mathbb{C}}
\newcommand{\CP}{\mathbb{C}P}
\newcommand{\SU}{\mathrm{SU}}

\newcommand{\Id}{\operatorname{Id}}
\newcommand{\Ind}{\operatorname{Ind}}
\newcommand{\Nul}{\operatorname{Nul}}
\newcommand{\II}{\mathrm{II}}
\newcommand{\dd}{\mathrm{d}}

\begin{document}

\begin{frontmatter}

\title{The Morse Index, Nullity and Jacobi Fields of Constant-Curvature Minimal 2-Spheres in Complex Projective Spaces}

\author[ustc]{Hongbin Cui}
\ead{cuihongbin@ustc.edu.cn}
\author[ustc]{Shuping Huang\corref{cor1}}
\cortext[cor1]{Corresponding author: Shuping Huang.}
\ead{hsp@mail.ustc.edu.cn}

\address[ustc]{School of Mathematical Sciences,
University of Science and Technology of China, 96 Jinzhai Road,
Hefei 230026, Anhui Province, China}

\begin{abstract}
We determine the Morse index and normal nullity of all constant-curvature
minimal two-spheres in complex projective spaces, identify the normal Jacobi
kernel as an $\SU(2)$-module, and prove that every normal Jacobi field is integrable.
The formulas include all totally geodesic extensions.  In a fixed
ambient space, the normal nullity is constant along each Veronese sequence.
\end{abstract}


\end{frontmatter}


\section{Introduction}
\label{sec:introduction}

We study the Morse index, normal nullity and Jacobi fields of
constant-curvature minimal two-spheres in complex projective spaces.
The Morse index measures instability under normal variations, and the
normal nullity is the dimension of the normal Jacobi kernel.  A related
question is whether every normal Jacobi field arises from a family of
minimal immersions.
A minimal two-sphere in $\CP^n$ is \emph{linearly full} if its
image is contained in no proper complex projective subspace.  Eells--Wood
developed the harmonic-sequence description of harmonic maps into $\CP^n$
\cite{EellsWood1983}.  Bando--Ohnita classified the linearly full
constant-curvature examples as the members of the Veronese sequence
\cite{Bando1987}; see also the treatment by
Bolton--Jensen--Rigoli--Woodward \cite{Bolton1988} and the equivariant
classification in \cite{Ohnita1990}.  Consequently,
up to an ambient isometry and reparametrization, every
constant-curvature minimal immersion $S^2\to\CP^N$ is a totally geodesic
extension of a linearly full Veronese member.

In the Veronese sequence indexed by $k=0,\ldots,n$, the holomorphic
member ($k=0$) and the anti-holomorphic member ($k=n$) are stable as
complex curves \cite{LawsonSimons1973,Simons1968}, whereas stability
results for harmonic maps into compact Hermitian symmetric spaces imply
that all members with $0<k<n$ are unstable
\cite{BurnsDeBartolomeis1988,BurnsEtAl1989,OhnitaUdagawa1987}.
Eells--Wood obtained general lower bounds for the energy index, later
sharpened by Oliver \cite{EellsWood1983,Oliver2020}.
Kimura computed the normal nullities of the Veronese maps with $k=0$ or
$k=n$ and their totally geodesic extensions \cite{Kimura1977}.
For the $k=1$ member in $\CP^2$, Ohnita obtained the index and nullity of the
embedded real projective plane \cite{Ohnita1987}, while Montiel--Urbano treated
its double-covering immersion \cite{MontielUrbano1997}.
For the Bor\r{u}vka spheres in round spheres, Ball--Madnick computed the
index and nullity and proved the integrability of Jacobi fields
\cite{BallMadnick2026}.
The complex projective spheres also arise in the Euclidean $\CP^N$ nonlinear
sigma model, where related stability and fluctuation questions have been
studied \cite{BurstallRawnsley1987,DabrowskiDunne2013,DinZakrzewski1980,DinZakrzewski1982}.

In complex projective spaces, Kimura proved that every normal Jacobi field
along a Veronese map with $k=0$ or $k=n$ is integrable through families
of K\"ahler embeddings
\cite{Kimura1977}.
Lemaire--Wood proved that every energy Jacobi field along a harmonic map
$S^2\to\CP^2$ is integrable.  They also obtained the corresponding result for
area Jacobi fields along minimal branched immersions by taking normal
components \cite{LemaireWood2002}.
For a minimal immersion of $S^2$, with the induced conformal structure on
the domain, the energy and area indices agree.  Normal projection
identifies the energy Jacobi kernel,
modulo the six-dimensional real space of infinitesimal conformal M\"obius
reparametrizations, with the normal Jacobi kernel
\cite{EjiriMicallef2008}; see also \cite{LemaireWood2002}.
Deformations of harmonic maps and the parameter spaces of harmonic
two-spheres have been studied in
\cite{CatenacciCornalbaReina1983,Crawford1997,GuestOhnita1993,Kawabe2013,LemaireWood1996}.

We determine the exact Morse index and normal nullity for every member of
the Veronese sequence in $\CP^n$ and every totally geodesic extension to
$\CP^N$.  The normal Jacobi kernel admits an explicit $\SU(2)$-module
decomposition, and all its elements are realized by deforming the rational
normal directrix.  Here the \emph{holomorphic directrix} is the initial
holomorphic curve generating the harmonic sequence by successive Gauss
transforms.  The Jacobi operator reduces to real symmetric matrix
blocks; their inertia is determined by an explicit congruence calculation
involving tridiagonal matrices whose spectra are computed directly.

For integers $n\ge1$ and $0\le k\le n$, put $\lambda\coloneqq n-2k$ and let
$\phi_{\lambda,n}:S^2\to\CP^n$ be the corresponding linearly full member of
the Veronese sequence.  For an integer $N\ge n$, let $\phi_{\lambda,n}^{N}$ be
its composition with a totally geodesic linear inclusion
$\CP^n\hookrightarrow\CP^N$.  When $n$ is even and $k=n/2$, the index and
nullity are computed on the covering sphere $S^2$, using all normal fields
along the two-fold covering immersion \cite{Bando1987,Bolton1988}.

\begin{thm}
	\label{thm:introduction-linearly-full}
	Let $n\ge1$ be an integer, let $k\in\{0,\ldots,n\}$, and put
	$\lambda\coloneqq n-2k$. Then
	\[
	\Ind(\phi_{\lambda,n})
	=2k(n-k)(n+1)
	=\frac{(n+1)(n^2-\lambda^2)}2,\qquad
	\Nul(\phi_{\lambda,n})=2(n-1)(n+3).
	\]
In particular, $\phi_{\lambda,n}$ is stable if and only if $k=0$ or $k=n$.
\end{thm}

\begin{thm}
	\label{thm:introduction-higher-codimension}
	Let $N\ge n\ge1$ be integers, let $k\in\{0,\ldots,n\}$, and put
	$\lambda\coloneqq n-2k$. Then
	\[
	\Ind_{\CP^N}(\phi_{\lambda,n}^{N})
	=2k(n-k)(N+1)
	=\frac{(N+1)(n^2-\lambda^2)}2,\qquad
	\Nul_{\CP^N}(\phi_{\lambda,n}^{N})
	=2\bigl((N+1)(n+1)-4\bigr).
	\]
\end{thm}

The formulas recover Kimura's nullities for $k=0$ or $k=n$
\cite{Kimura1977} and Montiel--Urbano's index and nullity for the
double-covering Veronese sphere with $k=1$ in $\CP^2$ \cite{MontielUrbano1997}.

For fixed $n$ and $N$, the normal nullity is independent of $k$, although
the index generally varies along the Veronese sequence.  The next theorem
addresses the integrability of normal Jacobi fields.

\begin{thm}
\label{thm:introduction-integrable-nullity}
Let $N\ge n\ge1$ be integers, let $k\in\{0,\ldots,n\}$, and put
$\lambda\coloneqq n-2k$.  Every normal Jacobi field $V$ along
$\phi_{\lambda,n}^{N}$ is integrable: there is a smooth family of minimal
immersions $\psi_t:S^2\to\CP^N$, obtained by deforming the rational normal
directrix, such that
\[
 \psi_0=\phi_{\lambda,n}^{N},\qquad
 \left(\left.\frac{\partial\psi_t}{\partial t}\right|_{t=0}\right)^\perp=V.
\]
\end{thm}

The paper is organized as follows.
Section~\ref{sec:preliminaries} records the representation and geometric
preliminaries.  Section~\ref{sec:Jacobi-operator} derives the matrix blocks
of the Jacobi operator in $\CP^n$.  Section~\ref{sec:index} computes the
index, nullity and normal Jacobi kernel in all codimensions and proves
integrability.
The proofs of Proposition~\ref{prop:Jacobi-matrices} and
Lemmas~\ref{lem:comparison-spectrum} and~\ref{lem:matrix-congruence}
are given in \ref{app:matrix-calculations}.

\section{Preliminaries}
\label{sec:preliminaries}

In this section, we recall some standard facts on the unitary representations
of $\SU(2)$, the geometry of the Veronese sequence and the second variation of area.

Let
\[
 \SU(2)\coloneqq\left\{g=\begin{pmatrix}\zeta&\eta\\-\bar\eta&\bar\zeta\end{pmatrix}:
 |\zeta|^2+|\eta|^2=1\right\},\qquad
 T\coloneqq\{t_u\coloneqq\operatorname{diag}(e^{\mathrm{i}u},e^{-\mathrm{i}u}):u\in\R\}.
\]
The Lie algebra $\mathfrak{su}(2)$ has the real basis
\[
 \varepsilon_1\coloneqq\begin{pmatrix}\mathrm{i}&0\\0&-\mathrm{i}\end{pmatrix},\qquad
 \varepsilon_2\coloneqq\begin{pmatrix}0&1\\-1&0\end{pmatrix},\qquad
 \varepsilon_3\coloneqq\begin{pmatrix}0&\mathrm{i}\\\mathrm{i}&0\end{pmatrix}.
\]
In particular, $t_u=\exp(u\varepsilon_1)$.
The Maurer--Cartan form and its structure equations are
\begin{equation}
\label{eq:Maurer-Cartan}
 \dd g\,g^{-1}=\begin{pmatrix}\mathrm{i}\omega&\varphi\\
 -\overline\varphi&-\mathrm{i}\omega\end{pmatrix},\qquad
 \dd(\mathrm{i}\omega)=-\varphi\wedge\overline\varphi,\qquad
 \dd\varphi=2\mathrm{i}\omega\wedge\varphi.
\end{equation}
Here $\omega$ is real and $\varphi=\omega_1+\mathrm{i}\omega_2$, so
$\dd g\,g^{-1}=\varepsilon_1\omega+\varepsilon_2\omega_1+\varepsilon_3\omega_2$.
The quotient $\SU(2)/T\coloneqq\{[g]=Tg:g\in\SU(2)\}$ is identified
with $\CP^1$ by $[g]\mapsto[\zeta:\eta]$, and hence with $S^2$.
We use the standard complex structure of $\CP^1$ and the right action
$[g]\cdot h\coloneqq[gh]$.  The metric
$\varphi\overline\varphi=\omega_1^2+\omega_2^2$ on $S^2$ has Gauss
curvature $4$, by \eqref{eq:Maurer-Cartan}.

Let $V_n$ be the space of homogeneous polynomials of degree $n$ in
$z_0,z_1$, with Hermitian inner product
\[
 \left\langle\sum_{k=0}^n c_kz_0^kz_1^{n-k},
 \sum_{k=0}^n d_kz_0^kz_1^{n-k}\right\rangle_{\C}
 \coloneqq\sum_{k=0}^n c_k\overline{d_k}\,k!(n-k)!.
\]
For $f\in V_n$, we take the right polynomial action
\begin{equation}
\label{eq:SU2-right-action}
 (f\rho_n(g))(z_0,z_1)\coloneqq f\bigl((z_0,z_1)g^{-1}\bigr)
 =f(\bar\zeta z_0+\bar\eta z_1,-\eta z_0+\zeta z_1)
\end{equation}
and the orthonormal weight basis
\begin{equation}
\label{eq:unitary-weight-basis}
 v_{n-2k,n}\coloneqq\frac{(-1)^kz_0^kz_1^{n-k}}{\sqrt{k!(n-k)!}},
 \qquad 0\le k\le n.
\end{equation}
Thus $V_n$ has complex dimension $n+1$ and weights
$n,n-2,\ldots,-n$, each with multiplicity one.

For $X\in\mathfrak{su}(2)$, put
\[
 f\,\dd\rho_n(X)\coloneqq\left.\frac{\dd}{\dd t}\right|_0 f\rho_n(\exp(tX)).
\]
For $\lambda=n-2k$, set
\[
 a_{\lambda,n}\coloneqq\tfrac12\sqrt{(n+1)^2-(\lambda-1)^2},\qquad
 b_{\lambda,n}\coloneqq\tfrac12\sqrt{(n+1)^2-(\lambda+1)^2}.
\]
A coefficient is set to zero if its source or target weight is absent;
terms involving absent weight vectors are omitted.
Direct differentiation gives the action of the real basis:
\[
 \begin{aligned}
 v_{\lambda,n}\,\dd\rho_n(\varepsilon_1)&=\mathrm{i}\lambda v_{\lambda,n},\\
 v_{\lambda,n}\,\dd\rho_n(\varepsilon_2)&=a_{\lambda,n}v_{\lambda-2,n}-b_{\lambda,n}v_{\lambda+2,n},\\
 v_{\lambda,n}\,\dd\rho_n(\varepsilon_3)&=\mathrm{i}\bigl(a_{\lambda,n}v_{\lambda-2,n}+b_{\lambda,n}v_{\lambda+2,n}\bigr).
 \end{aligned}
\]
Extend this action complex-linearly to $\mathfrak{sl}(2,\C)$, and put
\[
 H\coloneqq-\mathrm{i}\varepsilon_1=\begin{pmatrix}1&0\\0&-1\end{pmatrix},\qquad
 A\coloneqq\tfrac12(\varepsilon_2-\mathrm{i}\varepsilon_3)=\begin{pmatrix}0&1\\0&0\end{pmatrix},\qquad
 B\coloneqq-\tfrac12(\varepsilon_2+\mathrm{i}\varepsilon_3)=\begin{pmatrix}0&0\\1&0\end{pmatrix}.
\]
The same letters $H,A,B$ denote their represented matrices on $V_n$.
The formulas for $\varepsilon_i$ give
\begin{equation}
\label{eq:ladder-actions-section2}
 v_{\lambda,n}H=\lambda v_{\lambda,n},\qquad
 v_{\lambda,n}A=a_{\lambda,n}v_{\lambda-2,n},\qquad
 v_{\lambda,n}B=b_{\lambda,n}v_{\lambda+2,n}.
\end{equation}
Thus $a_{\lambda,n}^2=(k+1)(n-k)$,
$b_{\lambda,n}^2=k(n-k+1)$, and
$a_{\lambda+2,n}=b_{\lambda,n}$.
The represented operators satisfy
\[
 [H,A]=2A,\qquad [H,B]=-2B,\qquad [A,B]=H.
\]
The represented Maurer--Cartan equation is
\begin{equation}
\label{eq:represented-Maurer-Cartan}
 \dd\rho_n(g)\,\rho_n(g)^{-1}
 =\mathrm{i}\omega H+\varphi A-\overline\varphi B.
\end{equation}

For $\mu=n,n-2,\ldots,-n$, set $Z_\mu(g)\coloneqq v_{\mu,n}\rho_n(g)$.
These vectors define a local moving frame after choosing a local
section of the projection $\SU(2)\to S^2$.
The ladder identities \eqref{eq:ladder-actions-section2} and
\eqref{eq:represented-Maurer-Cartan} give
\begin{equation}
\label{eq:dZ-mu}
 \dd Z_\mu=\mu\mathrm{i}\omega Z_\mu
       +a_{\mu,n}\varphi Z_{\mu-2}
       -b_{\mu,n}\overline\varphi Z_{\mu+2}.
\end{equation}

The representations $V_s$, $s\ge0$, are the irreducible unitary
representations of $\SU(2)$.  Every finite-dimensional unitary
representation is an orthogonal sum of these representations
\cite{BrockerTomDieck1985}.

Using the orthonormal weight basis, identify $V_n$ with $\C^{n+1}$ and
write $[v]$ for the complex line spanned by a nonzero vector $v$.
For $n\ge1$ and $0\le k\le n$, set $\lambda\coloneqq n-2k$.
Since $v_{\lambda,n}\rho_n(t_u)=e^{\mathrm{i}\lambda u}v_{\lambda,n}$,
the projective orbit
\[
 \phi_{\lambda,n}:S^2\simeq\SU(2)/T\longrightarrow\CP^n,\qquad
 [g]\longmapsto[v_{\lambda,n}\rho_n(g)]
\]
is well defined.  These maps form the Veronese harmonic sequence and are
linearly full conformal minimal immersions.  Write
\[
 D_{\lambda,n}\coloneqq\frac{n(n+2)-\lambda^2}{2}=n+2k(n-k).
\]
The ladder coefficients satisfy
\[
 a_{\lambda,n}^2+b_{\lambda,n}^2=D_{\lambda,n},\qquad
 a_{\lambda,n}^2-b_{\lambda,n}^2=\lambda.
\]
Using the local unit lift $Z_\lambda$, we represent tangent vectors
to $\CP^n$ along $\phi_{\lambda,n}$ by vectors $X\in\C^{n+1}$
satisfying $\langle X,Z_\lambda\rangle_\C=0$.
For the Fubini--Study metric of holomorphic sectional curvature $1$,
the metric on these representatives is
\[
 g_{\mathrm{FS}}(X,Y)=4\operatorname{Re}\langle X,Y\rangle_\C,
\]
and the ambient complex structure is multiplication by $\mathrm{i}$.
Unless a subscript $\C$ is shown, inner products are taken with respect
to this real metric. The induced metric and Gauss curvature are
\begin{equation}
\label{eq:classical-model-metric-curvature}
 \phi_{\lambda,n}^*g_{\mathrm{FS}}=4D_{\lambda,n}\varphi\overline\varphi,
 \qquad K=\frac1{D_{\lambda,n}}=\frac1{n+2k(n-k)}.
\end{equation}
For $k=0$, $\phi_{n,n}$ is holomorphic; for $k=n$, $\phi_{-n,n}$ is
anti-holomorphic.

\begin{thm}[Bando--Ohnita \cite{Bando1987}; Bolton--Jensen--Rigoli--Woodward \cite{Bolton1988}]
\label{thm:BO-classification}
Let $\psi:S^2\to\CP^n$ be a linearly full minimal immersion, with $S^2$
endowed with the induced conformal structure.  If its
induced metric has constant Gauss curvature, then, up to an ambient
isometry and a conformal or anti-conformal reparametrization, it is
$\phi_{n-2k,n}$ for some $0\le k\le n$.
Its Gauss curvature is $1/(n+2k(n-k))$.
\end{thm}

We use the following line bundles to describe the normal coefficients.
For $r\in\tfrac12\mathbb Z$, define the complex line bundle
\[
 L_r\coloneqq(\SU(2)\times\C)/{\sim},\qquad
 (g,\xi)\sim(t_ug,e^{-2r\mathrm{i}u}\xi),
\]
with projection $[g,\xi]\mapsto[g]=Tg$.
The right $\SU(2)$-action on $L_r$ is
$[g,\xi]\cdot h\coloneqq[gh,\xi]$, $h\in\SU(2)$.
A smooth section of $L_r$ is represented by a smooth function
$f:\SU(2)\to\C$ satisfying
\[
 f(t_ug)=e^{-2r\mathrm{i}u}f(g),\qquad
 \text{with section }[g]\longmapsto[g,f(g)].
\]
The equivariance condition makes the section independent of the
representative $g$.  We identify sections with these functions and write
$\Gamma(L_r)$ and $L^2(L_r)$ for the spaces of smooth and square-integrable
sections, respectively, using the fibre metric induced by $|\xi|^2$
and the area measure of the induced metric.
For a smooth function $f$ on $\SU(2)$ and $X\in\mathfrak{su}(2)$, define
\[
 D_Xf(g)\coloneqq\left.\frac{\dd}{\dd t}\right|_{t=0}
 f\bigl(\exp(tX)g\bigr),
\]
and extend complex-linearly in $X$.  In terms of the real basis above,
\[
 D_A=\tfrac12\bigl(D_{\varepsilon_2}-\mathrm{i}D_{\varepsilon_3}\bigr),\qquad
 D_B=-\tfrac12\bigl(D_{\varepsilon_2}+\mathrm{i}D_{\varepsilon_3}\bigr).
\]
The Maurer--Cartan coframe in \eqref{eq:Maurer-Cartan} gives
\[
 \dd f=(D_{\varepsilon_1}f)\omega
       +(D_{\varepsilon_2}f)\omega_1+(D_{\varepsilon_3}f)\omega_2.
\]
For a section of $L_r$, equivariance gives
$D_{\varepsilon_1}f=-2r\mathrm{i}f$.
The horizontal distribution $\omega=0$ induces a unitary connection
on $L_r$: removing the vertical term $-2r\mathrm{i}f\omega$ and using
$\varphi=\omega_1+\mathrm{i}\omega_2$ gives
\begin{equation}
\label{eq:horizontal-derivatives}
 \nabla f\coloneqq\dd f+2r\mathrm{i}\omega f
 =(D_Af)\varphi-(D_Bf)\overline\varphi.
\end{equation}
Thus $D_A:\Gamma(L_r)\to\Gamma(L_{r+1})$ and
$D_B:\Gamma(L_r)\to\Gamma(L_{r-1})$.

Let $\Delta\coloneqq-\nabla^*\nabla$ be the connection Laplacian on $L_r$ for the
metric $4D_{\lambda,n}\varphi\overline\varphi$.
For $i=2,3$, the curves $t\mapsto\exp(t\varepsilon_i)g$ are horizontal
lifts of geodesics with orthogonal initial velocities of length
$2\sqrt{D_{\lambda,n}}$.
Taking the trace of the second covariant derivative along these curves gives
\begin{equation}
\label{eq:line-laplacian-horizontal}
 \Delta=\frac{D_{\varepsilon_2}^{\,2}+D_{\varepsilon_3}^{\,2}}{4D_{\lambda,n}}
       =-\frac{D_AD_B+D_BD_A}{2D_{\lambda,n}}.
\end{equation}

For a minimal immersion of a closed surface, let $\II$ and $\nabla^\perp$
denote its second fundamental form and normal connection.
For a local orthonormal tangent frame $e_1,e_2$, define
\[
 \mathcal B V\coloneqq\sum_{i,j}\langle\II(e_i,e_j),V\rangle\II(e_i,e_j),\qquad
 \mathcal R V\coloneqq\sum_i(\overline R(V,e_i)e_i)^\perp,
\]
where $\overline R$ is the ambient curvature tensor and $\perp$ denotes
normal projection.
We use the trace Laplacian $\Delta^\perp\coloneqq-(\nabla^\perp)^*\nabla^\perp$.
The second variation formula \cite{EellsLemaire1983,Simons1968} gives
the normal Jacobi operator $J$ and the area Hessian $Q$ on real normal fields:
\begin{equation}
\label{eq:Jacobi-operator-definition}
 J=\Delta^\perp+\mathcal B+\mathcal R,\qquad
 Q(V,W)=-\int\langle JV,W\rangle\,\dd A.
\end{equation}
Here $\dd A$ is the induced area measure.  The Morse index counts the
positive eigenvalues of $J$ with real multiplicity, and the normal nullity
is $\dim_\R\ker J$.

\section{The Jacobi operator}
\label{sec:Jacobi-operator}

In this section, we compute the normal Jacobi operator of the linearly full
Veronese spheres and reduce it to finite real symmetric matrices.

Throughout this section and Section~\ref{subsec:linearly-full-index}, we fix
$n\ge1$ and $0\le k\le n$ and use the notation of Section~\ref{sec:preliminaries}.

\subsection{The Jacobi operator in normal coefficients}
\label{subsec:normal-connection}

By \eqref{eq:dZ-mu}, modulo $\C Z_\lambda$ one has
\[
 \dd Z_\lambda=a_{\lambda,n}Z_{\lambda-2}\varphi-b_{\lambda,n}Z_{\lambda+2}\overline\varphi.
\]
Consequently, an oriented orthonormal tangent frame is
\begin{equation}
\label{eq:tangent-frame}
 e_1\coloneqq\frac{a_{\lambda,n}Z_{\lambda-2}-b_{\lambda,n}Z_{\lambda+2}}{2\sqrt{D_{\lambda,n}}},\qquad
 e_2\coloneqq\frac{\mathrm{i}(a_{\lambda,n}Z_{\lambda-2}+b_{\lambda,n}Z_{\lambda+2})}{2\sqrt{D_{\lambda,n}}}.
\end{equation}
We also write $e_i$ for the corresponding domain tangent vectors under
$\dd\phi_{\lambda,n}$.
If $0<k<n$, the orthogonal complement of the tangent plane in
$\C Z_{\lambda-2}\oplus\C Z_{\lambda+2}$ has orthonormal frame
\begin{equation}
\label{eq:first-normal-frame}
 \nu_1\coloneqq\frac{b_{\lambda,n}Z_{\lambda-2}+a_{\lambda,n}Z_{\lambda+2}}{2\sqrt{D_{\lambda,n}}},\qquad
 \nu_2\coloneqq\frac{\mathrm{i}(-b_{\lambda,n}Z_{\lambda-2}+a_{\lambda,n}Z_{\lambda+2})}{2\sqrt{D_{\lambda,n}}}.
\end{equation}
Equations~\eqref{eq:tangent-frame}--\eqref{eq:first-normal-frame} give
the orthogonal normal splitting
\[
 N_{\phi_{\lambda,n}}S^2
 =\operatorname{span}_{\R}\{\nu_1,\nu_2\}
 \oplus\bigoplus_{r=2}^{n-k}(\C Z_{\lambda-2r})_\R
 \oplus\bigoplus_{r=2}^{k}(\C Z_{\lambda+2r})_\R.
\]
Here $(\C Z)_\R$ denotes the underlying real two-plane bundle.
The local oriented orthonormal frames are, respectively,
\[
 (\nu_1,\nu_2),\qquad
 \bigl(Z_{\lambda-2r}/2,-\mathrm{i}Z_{\lambda-2r}/2\bigr),\qquad
 \bigl(Z_{\lambda+2r}/2,\mathrm{i}Z_{\lambda+2r}/2\bigr).
\]
For $k=0$ or $k=n$, the central summand is omitted.

Put $\alpha\coloneqq a_{\lambda,n}\,a_{\lambda-2,n}$ and $\beta\coloneqq b_{\lambda,n}\,b_{\lambda+2,n}$.
Differentiating \eqref{eq:tangent-frame} with \eqref{eq:dZ-mu} and taking
normal components gives
\begin{equation}
\label{eq:II-frame-formula}
 \II(e_1,e_1)=\frac{\alpha Z_{\lambda-4}+\beta Z_{\lambda+4}}{4D_{\lambda,n}},\qquad
 \II(e_1,e_2)=\frac{\mathrm{i}(\alpha Z_{\lambda-4}-\beta Z_{\lambda+4})}{4D_{\lambda,n}},\qquad
 \II(e_2,e_2)=-\II(e_1,e_1).
\end{equation}

Using the normal frames above, every real normal field $V$ has a unique
expansion
\begin{equation}
\label{eq:normal-field-complex-coordinates}
 V=\frac{b_{\lambda,n}\overline{A_1}Z_{\lambda-2}+a_{\lambda,n}A_1Z_{\lambda+2}}{2\sqrt{D_{\lambda,n}}}
  +\frac12\sum_{r=2}^{n-k}\overline{A_r^-}Z_{\lambda-2r}
  +\frac12\sum_{r=2}^{k}A_r^+Z_{\lambda+2r}.
\end{equation}
When $k=0$ or $k=n$, all $A_1$-terms are omitted.
The coefficients are complex-valued functions on $\SU(2)$, given explicitly by
\[
 A_r^-=\left\langle V,\frac{Z_{\lambda-2r}}2\right\rangle
 +\mathrm{i}\left\langle V,-\frac{\mathrm{i}Z_{\lambda-2r}}2\right\rangle,\qquad
 A_r^+=\left\langle V,\frac{Z_{\lambda+2r}}2\right\rangle
 +\mathrm{i}\left\langle V,\frac{\mathrm{i}Z_{\lambda+2r}}2\right\rangle.
\]
For $0<k<n$, $A_1=\langle V,\nu_1\rangle+\mathrm{i}\langle V,\nu_2\rangle$.
The minus sign in the second lower frame vector accounts for the
conjugate coefficient in \eqref{eq:normal-field-complex-coordinates}.
The squared norm is
\[
 |V|^2=|A_1|^2+\sum_r|A_r^-|^2+\sum_r|A_r^+|^2.
\]
We use the complex structure on the normal bundle induced by
multiplication of these coefficients by $\mathrm{i}$.

Relative to the lift $Z_\lambda$, the directions $Z_{\lambda\pm2r}$
have weights $\pm2r$.  With the conjugation on the lower branch in
\eqref{eq:normal-field-complex-coordinates}, the normal coefficients satisfy
\begin{equation}
\label{eq:normal-coefficient-equivariance}
 A_1(t_ug)=e^{-2\mathrm{i}u}A_1(g),\qquad
 A_r^\pm(t_ug)=e^{-2\mathrm{i}ru}A_r^\pm(g).
\end{equation}

Thus $A_1\in\Gamma(L_1)$ and $A_r^\pm\in\Gamma(L_r)$.
A coefficient in $\Gamma(L_r)$ has line index $r$.
We let $\nabla$ act componentwise as in \eqref{eq:horizontal-derivatives}.
Write $(\nabla^\perp V)_1$ and $(\nabla^\perp V)_r^\pm$ for the
complex-valued one-form coefficients of $\nabla^\perp V$ in the normal-coordinate
convention \eqref{eq:normal-field-complex-coordinates}.
Differentiating \eqref{eq:normal-field-complex-coordinates} using
\eqref{eq:dZ-mu}, subtracting the connection term $\lambda\mathrm{i}\omega V$
arising from the unit lift $Z_\lambda$, and taking normal components gives
\begin{equation}
\label{eq:normal-connection-components}
 \begin{aligned}
 (\nabla^\perp V)_1
 &=\nabla A_1
 -\frac{b_{\lambda,n}a_{\lambda-2,n}}{\sqrt{D_{\lambda,n}}}A_2^-\varphi
 +\frac{a_{\lambda,n}b_{\lambda+2,n}}{\sqrt{D_{\lambda,n}}}A_2^+\varphi,\\[5pt]
 (\nabla^\perp V)_2^-
 &=\nabla A_2^-
 +\frac{b_{\lambda,n}a_{\lambda-2,n}}{\sqrt{D_{\lambda,n}}}A_1\overline\varphi
 -a_{\lambda-4,n}A_3^-\varphi,\\[5pt]
 (\nabla^\perp V)_2^+
 &=\nabla A_2^+
 -\frac{a_{\lambda,n}b_{\lambda+2,n}}{\sqrt{D_{\lambda,n}}}A_1\overline\varphi
 +b_{\lambda+4,n}A_3^+\varphi,\\[5pt]
 (\nabla^\perp V)_r^-
 &=\nabla A_r^-
 +a_{\lambda-2r+2,n}A_{r-1}^-\overline\varphi
 -a_{\lambda-2r,n}A_{r+1}^-\varphi,\\[5pt]
 (\nabla^\perp V)_r^+
 &=\nabla A_r^+
 -b_{\lambda+2r-2,n}A_{r-1}^+\overline\varphi
 +b_{\lambda+2r,n}A_{r+1}^+\varphi.
 \end{aligned}
\end{equation}
The last two lines apply for $3\le r\le n-k$ and $3\le r\le k$,
respectively.  Absent components and their terms are omitted.

Let $u_1,\ldots,u_{n-1}$ be the normal coefficients in the order
$A_1,A_2^-,\ldots,A_{n-k}^-,A_2^+,\ldots,A_k^+$, omitting absent components,
and write $[\nabla^\perp V]_i$ for the corresponding one-form coefficients.
Equation~\eqref{eq:normal-connection-components} defines the constant real
matrix $C=(C_{ij})$ by
\[
 \left[\nabla^\perp V\right]_i=\nabla u_i+
 \sum_{j=1}^{n-1}\left(C_{ij}\bar\varphi-C_{ji}\varphi\right)u_j,
 \qquad 1\le i\le n-1.
\]
All coefficient matrices below use this order.

The operators $D_A,D_B$ and $\Delta$ act componentwise, using the line
index of each input.  Fix a lift $g\in\SU(2)$, and let $X_a$ be the initial
velocity of $t\mapsto[\exp(t\varepsilon_a)g]$, $a=2,3$.
Evaluating the preceding one-form identity along the corresponding
horizontal lifts gives, for $1\le i\le n-1$,
\[
 \begin{aligned}
 [\nabla^\perp_{X_2}V]_i
 &=D_{\varepsilon_2}u_i+\sum_{j=1}^{n-1}(C_{ij}-C_{ji})u_j,\\[4pt]
 [\nabla^\perp_{X_3}V]_i
 &=D_{\varepsilon_3}u_i-\mathrm{i}\sum_{j=1}^{n-1}(C_{ij}+C_{ji})u_j.
 \end{aligned}
\]
The projected curves are geodesics, and $X_2,X_3$ are orthogonal with
length $2\sqrt{D_{\lambda,n}}$.  Thus, as in
\eqref{eq:line-laplacian-horizontal}, taking the trace of the second
covariant derivative gives, in normal coefficients,
\[
 \Delta^\perp
 =\frac{(D_{\varepsilon_2}+C-C^*)^2+
        (D_{\varepsilon_3}-\mathrm{i}(C+C^*))^2}{4D_{\lambda,n}}.
\]
In this coefficient representation, we use the same letters
$\mathcal B,\mathcal R,J$ for the corresponding operators.
Expanding the squares with $D_{\varepsilon_2}=D_A-D_B$ and
$D_{\varepsilon_3}=\mathrm{i}(D_A+D_B)$, and adding the two terms
$\mathcal B$ and $\mathcal R$, gives
\begin{equation}
\label{eq:Jacobi-differential-coefficients}
 J=\Delta+\frac{CD_A+C^*D_B}{D_{\lambda,n}}
       -\frac{C^*C+CC^*}{2D_{\lambda,n}}
       +\mathcal B+\mathcal R.
\end{equation}

In the chosen normal coefficients, $\mathcal B$ and $\mathcal R$ are
constant real matrices preserving the line index
(see \ref{app:matrix-calculations}), so the operator in
\eqref{eq:Jacobi-differential-coefficients} is complex-linear.

\subsection{Reduction to finite matrices}
\label{subsec:coefficient-modes}

We decompose the coefficient spaces $L^2(L_r)$ into irreducible
$\SU(2)$-representations.
For $r\in\tfrac12\mathbb Z$, right translation acts on $L^2(L_r)$ by
$(f\cdot h)(g)\coloneqq f(gh^{-1})$.  The action of $T$ on the fibre
$(L_r)_{[I]}$ is $[I,\xi]\cdot t_u=[I,e^{2r\mathrm{i}u}\xi]$.
Frobenius reciprocity gives the multiplicity formula
\cite{BrockerTomDieck1985,Helgason1984}
\[
 \dim_\C\operatorname{Hom}_{\SU(2)}(V_s,L^2(L_r))
 =\dim_\C\operatorname{Hom}_T(V_s,(L_r)_{[I]})
 =\begin{cases}
 1,&s\ge2|r|,\quad s\equiv2r\pmod2,\\
 0,&\text{otherwise}.
 \end{cases}
\]
Here $\operatorname{Hom}_G$ denotes the space of $G$-equivariant
complex-linear maps.  The last equality follows from the weights
$s,s-2,\ldots,-s$ of $V_s$, each of multiplicity one.
For every $s$ satisfying these conditions, define
\[
 Y_{r,s}^{j}(g)\coloneqq
 \overline{\bigl(v_{2r,s}\rho_s(g)\bigr)_j},\qquad
 g\in\SU(2),\quad 1\le j\le s+1.
\]
The subscript $j$ denotes the $j$th coordinate in the ordered weight
basis of $V_s$.  Since $v_{2r,s}\rho_s(t_u)=e^{2r\mathrm{i}u}v_{2r,s}$,
\[
 Y_{r,s}^{j}(t_ug)=e^{-2r\mathrm{i}u}Y_{r,s}^{j}(g),
\]
so $Y_{r,s}^{j}$ is a section of $L_r$.  Unitarity of $\rho_s$ also gives
\[
 Y_{r,s}^{j}(gh^{-1})
 =\sum_{q=1}^{s+1}\rho_s(h)_{jq}\,Y_{r,s}^{q}(g),\qquad h\in\SU(2).
\]
Thus these functions transform in the same way as the weight basis of
$V_s$.  For fixed $r$ and admissible $s,s'$, Schur orthogonality gives
\[
 \int_{S^2}Y_{r,s}^{j}\,\overline{Y_{r,s'}^{j'}}\,\dd A
 =\frac{4\pi D_{\lambda,n}}{s+1}\,\delta_{ss'}\delta_{jj'}.
\]
Together with the multiplicity formula above, this gives the complete
orthogonal decomposition
\begin{equation}
\label{eq:coefficient-modes}
 L^2(L_r)=\widehat{\bigoplus}_{\substack{s\ge2|r|\\s\equiv2r\ (2)}}
 \operatorname{span}_{\C}\left\{
 Y_{r,s}^{j}:
 1\le j\le s+1\right\}.
\end{equation}
Each summand is isomorphic to $V_s$; the hat denotes the Hilbert
orthogonal sum, with convergence in $L^2$.

The moving-frame formula \eqref{eq:dZ-mu}, with $n=s$ and $\mu=2r$,
gives after coordinatewise conjugation
\[
 \dd Y_{r,s}^{j}+2r\mathrm{i}\omega Y_{r,s}^{j}
 =-b_{2r,s}Y_{r+1,s}^{j}\varphi
  +a_{2r,s}Y_{r-1,s}^{j}\overline\varphi.
\]
Comparison with \eqref{eq:horizontal-derivatives} gives
\begin{equation}
\label{eq:coefficient-ladder}
 D_AY_{r,s}^{j}=-b_{2r,s}Y_{r+1,s}^{j},\qquad
 D_BY_{r,s}^{j}=-a_{2r,s}Y_{r-1,s}^{j}.
\end{equation}
Terms beyond the weight range are zero.  The ladder identities give
\[
 D_AD_BY_{r,s}^{j}=a_{2r,s}^{\,2}Y_{r,s}^{j},\qquad
 D_BD_AY_{r,s}^{j}=b_{2r,s}^{\,2}Y_{r,s}^{j}.
\]
Hence \eqref{eq:line-laplacian-horizontal} yields
\begin{equation}
\label{eq:line-bundle-spectrum}
 \Delta f=-\frac{a_{2r,s}^{\,2}+b_{2r,s}^{\,2}}{2D_{\lambda,n}}f
 =-\frac{s(s+2)-4r^2}{4D_{\lambda,n}}f
\end{equation}
for every $f$ in the summand indexed by $s$ in
\eqref{eq:coefficient-modes}.

Since the normal line indices are positive integers,
\eqref{eq:coefficient-modes} shows that only the representations $V_{2\ell}$
occur, with $\ell\ge r$ for a coefficient in $L_r$.
On these summands, the coefficients in the ladder identities
\eqref{eq:coefficient-ladder} are
\begin{equation}
\label{eq:normal-mode-coefficients}
 a_{2r,2\ell}=\sqrt{\ell(\ell+1)-r(r-1)},\qquad
 b_{2r,2\ell}=\sqrt{\ell(\ell+1)-r(r+1)},\qquad 1\le r\le\ell.
\end{equation}
For fixed $\ell\ge1$, put
\[
 h_\ell\coloneqq\min(n-k,\ell),\qquad
 m_\ell\coloneqq h_\ell+\min(k,\ell)-1.
\]
Exactly $m_\ell$ normal coefficients have line index at most $\ell$.
We retain these coefficients in their original relative order and renumber
them consecutively.
When $m_\ell=0$, there is no normal block for this value of $\ell$.
Otherwise, fix $1\le j\le2\ell+1$ and a constant vector
$c=(c_1,\ldots,c_{m_\ell})^{\mathrm T}\in\C^{m_\ell}$.
For $0<k<n$, define a normal field $V$ by
\eqref{eq:normal-field-complex-coordinates} with
\[
 \begin{aligned}
 A_1&=c_1Y_{1,2\ell}^{j},\\[3pt]
 A_r^-&=(-1)^{r-1}c_rY_{r,2\ell}^{j}
       &&(2\le r\le h_\ell),\\[3pt]
 A_r^+&=c_{h_\ell+r-1}Y_{r,2\ell}^{j}
       &&(2\le r\le\min(k,\ell)).
 \end{aligned}
\]
For $k=0$ and $k=n$, respectively, use
\[
 A_r^-=(-1)^{r-1}c_{r-1}Y_{r,2\ell}^{j},\qquad
 A_r^+=c_{r-1}Y_{r,2\ell}^{j}\qquad(2\le r\le\min(n,\ell)).
\]
All other normal coefficients are zero.  Let $E_p$ be the field
obtained by setting $c_p=1$ and all other constants equal to zero.

The Laplacian in \eqref{eq:Jacobi-differential-coefficients} acts
diagonally by \eqref{eq:line-bundle-spectrum}; the constant order-zero
matrices preserve $r$, and the first-order terms couple adjacent line
indices by \eqref{eq:coefficient-ladder} and
\eqref{eq:normal-mode-coefficients}.
These operations preserve $\ell$ and $j$, and $b_{2\ell,2\ell}=0$
ensures that no component of line index $\ell+1$ occurs.
Thus $J$ preserves the complex span of $E_1,\ldots,E_{m_\ell}$.
Its matrix $M_\ell$ satisfies
\begin{equation}
\label{eq:multiplicity-ordering}
 V=\sum_{p=1}^{m_\ell}c_pE_p,\qquad
 JV=\sum_{p=1}^{m_\ell}(M_\ell c)_pE_p,\qquad c\in\C^{m_\ell}.
\end{equation}
Thus $JV=\mu V$ is equivalent to $M_\ell c=\mu c$.
The lower-branch signs make the adjacent off-diagonal entries of $M_\ell$
along that branch positive.

For fixed $\ell$, the same $M_\ell$ represents $J$ for all $j$.
The spaces constructed for distinct $j$ are mutually orthogonal,
and each basis field satisfies
$\|E_p\|_{L^2}^2=4\pi D_{\lambda,n}/(2\ell+1)$.
Varying $j$ in each normal coordinate gives one copy of $V_{2\ell}$.
With the coefficient complex structure, completeness of
\eqref{eq:coefficient-modes} yields
\[
 L^2(N_{\phi_{\lambda,n}}S^2)
 \cong\widehat{\bigoplus}_{\ell=1}^{\infty}(V_{2\ell})^{\oplus m_\ell}.
\]

\begin{prop}[Matrix blocks of the Jacobi operator]
\label{prop:Jacobi-matrices}
For $\ell\ge1$ with $m_\ell>0$, the matrix $M_\ell$ in
\eqref{eq:multiplicity-ordering} is real symmetric.
For $0<k<n$, its diagonal entries are
\begin{equation}
\label{eq:Jacobi-matrix-diagonals}
 \begin{aligned}
 (M_\ell)_{1,1}&=\frac{D_{\lambda,n}+2-\ell(\ell+1)}{D_{\lambda,n}},\\[3pt]
 (M_\ell)_{r,r}&=\frac{r(2r-\lambda)-\ell(\ell+1)}{D_{\lambda,n}}
                         +\frac{\delta_{r2}\alpha^2}{D_{\lambda,n}^2}
                         &&(2\le r\le h_\ell),\\[3pt]
 (M_\ell)_{h_\ell+r-1,h_\ell+r-1}&=\frac{r(2r+\lambda)-\ell(\ell+1)}{D_{\lambda,n}}
                         +\frac{\delta_{r2}\beta^2}{D_{\lambda,n}^2}
                         &&(2\le r\le\min(k,\ell)).
 \end{aligned}
\end{equation}
Here $\delta_{r2}$ is the Kronecker symbol.

\Needspace{12\baselineskip}
Its off-diagonal entries, up to symmetry, are
\begin{equation}
\label{eq:Jacobi-matrix-edges}
 \begin{alignedat}{2}
 (M_\ell)_{1,2}&=\frac{b_{\lambda,n}\,a_{\lambda-2,n}
                         \sqrt{\ell(\ell+1)-2}}{D_{\lambda,n}^{3/2}},&\qquad
 (M_\ell)_{1,h_\ell+1}&=\frac{a_{\lambda,n}\,b_{\lambda+2,n}
                         \sqrt{\ell(\ell+1)-2}}{D_{\lambda,n}^{3/2}},\\[5pt]
 (M_\ell)_{r,r+1}&=\frac{a_{\lambda-2r,n}
                         \sqrt{\ell(\ell+1)-r(r+1)}}{D_{\lambda,n}},&\qquad
 (M_\ell)_{h_\ell+r-1,h_\ell+r}&=\frac{b_{\lambda+2r,n}
                         \sqrt{\ell(\ell+1)-r(r+1)}}{D_{\lambda,n}},\\[5pt]
 (M_\ell)_{2,h_\ell+1}&=-\frac{\alpha\beta}{D_{\lambda,n}^2}.
 \end{alignedat}
\end{equation}
The first two entries occur when $h_\ell\ge2$ and $\min(k,\ell)\ge2$,
respectively, and the last occurs when both inequalities hold.
In the third and fourth entries, the ranges are $2\le r<h_\ell$
and $2\le r<\min(k,\ell)$, respectively.

For $k=0$ or $k=n$, row $r-1$ corresponds to $A_r^-$ or $A_r^+$,
respectively, and the entries are
\begin{equation}
\label{eq:Jacobi-matrix-endpoints}
 \begin{aligned}
 (M_\ell)_{r-1,r-1}
 &=\frac{r(2r-n)-\ell(\ell+1)}{n}
   +\frac{2(n-1)}{n}\delta_{r2}
   &&(2\le r\le\min(n,\ell)),\\[3pt]
 (M_\ell)_{r-1,r}
 &=\frac{\sqrt{(r+1)(n-r)\bigl(\ell(\ell+1)-r(r+1)\bigr)}}{n}
   &&(2\le r<\min(n,\ell)).
 \end{aligned}
\end{equation}
In all cases, the remaining off-diagonal entries are zero.

These blocks exhaust the spectrum of $J$, and for every $\mu\in\R$,
\[
 \dim_\R\ker(J-\mu\Id)
 =2\sum_{\substack{\ell\ge1\\m_\ell>0}}(2\ell+1)
       \dim_\C\ker(M_\ell-\mu\Id).
\]
For $\ell=1$, one has $M_1=(1)$ if $0<k<n$, and $m_1=0$ if $k=0,n$.
\end{prop}

The proof is given in \ref{app:matrix-calculations}.

\section{Index, nullity and Jacobi fields}
\label{sec:index}

We first compute the index, nullity and normal Jacobi kernel for the
linearly full spheres in $\CP^n$.  We then treat their totally geodesic
extensions to $\CP^N$ and prove that every normal Jacobi field is integrable.

\subsection{The linearly full case}
\label{subsec:linearly-full-index}

To determine the inertia of $M_\ell$, we introduce a tridiagonal comparison
matrix with explicitly computable spectrum, congruent to the direct sum of
$-D_{\lambda,n}M_\ell$ and two scalar blocks.
For $\ell\ge1$, define a real symmetric tridiagonal matrix $T_\ell$
of order $m_\ell+2$, with row $i$ corresponding to the integer
\begin{equation}
\label{eq:comparison-weight-basis}
 d=i-h_\ell-1,\qquad 1\le i\le m_\ell+2.
\end{equation}
Thus $d$ runs from $-h_\ell$ to $\min(k,\ell)$.
The diagonal and upper off-diagonal entries are
\begin{equation}
\label{eq:comparison-matrix-entries}
 \begin{aligned}
 (T_\ell)_{i,i}&\coloneqq\ell(\ell+1)-d(2d+\lambda)
   &&(1\le i\le m_\ell+2),\\
 (T_\ell)_{i,i+1}&\coloneqq-\sqrt{(\ell-d)(\ell+d+1)(k-d)(n-k+d+1)}
   &&(1\le i\le m_\ell+1).
 \end{aligned}
\end{equation}
The entries below the diagonal are determined by symmetry, and all
other entries are zero.

The following two results, proved in \ref{app:matrix-calculations},
give the spectrum of $T_\ell$ and its relation to $M_\ell$.

\begin{lem}[Spectrum of the comparison matrix]
\label{lem:comparison-spectrum}
For $\ell\ge1$, the eigenvalues of $T_\ell$ are simple and are
\begin{equation}
\label{eq:T-spectrum}
 \operatorname{Spec}(T_\ell)
 =\{(\ell-j)(n+\ell-j+1):j=0,\ldots,m_\ell+1\}.
\end{equation}
\end{lem}

\begin{lem}[An explicit congruence]
\label{lem:matrix-congruence}
For every $\ell\ge1$, there is an invertible real matrix $P$ such that
\begin{equation}
\label{eq:matrix-congruence}
 P^{\mathrm{T}}T_\ell P
 =\operatorname{diag}\bigl(\ell(\ell+1),\ell(\ell+1)-2,-D_{\lambda,n}M_\ell\bigr).
\end{equation}
The normal block is omitted for $n=1$, and also for $\ell=1$ with
$k\in\{0,n\}$.
For $\ell=1$ and $0<k<n$, it is $-D_{\lambda,n}M_1=(-D_{\lambda,n})$.
\end{lem}

\begin{proof}[Proof of Theorem~\ref{thm:introduction-linearly-full}]
For $2\le\ell\le n$, one has $\ell\le m_\ell+1\le n$.
The factor $n+\ell-j+1$ in \eqref{eq:T-spectrum} is positive, so the
sign of each eigenvalue is the sign of $\ell-j$.
Since $D_{\lambda,n}>0$ and both scalar blocks in
\eqref{eq:matrix-congruence} are positive, Sylvester's law of inertia
\cite{HornJohnson2013} gives
\[
 \begin{array}{c|ccc}
 \text{matrix}&\text{positive}&\text{zero}&\text{negative}\\[2pt]\hline
 T_\ell&\ell&1&m_\ell+1-\ell\\[2pt]
 -D_{\lambda,n}M_\ell&\ell-2&1&m_\ell+1-\ell\\[2pt]
 M_\ell&m_\ell+1-\ell&1&\ell-2
 \end{array}
\]
If $\ell>n$, then $j\le m_\ell+1\le n<\ell$, so $T_\ell$ is
positive definite and every nonempty $M_\ell$ is negative definite.
For $\ell=1$, Proposition~\ref{prop:Jacobi-matrices} gives $M_1=(1)$
when $0<k<n$, with no normal block when $k=0,n$.

For $1\le\ell\le n-1$, the number of positive eigenvalues of $M_\ell$ is
\[
 m_\ell+1-\ell
 =\min(k,\ell)-\max\{0,\ell-(n-k)\}.
\]
This is the number of integers $r$ satisfying
$1\le r\le k$ and $r\le\ell<n-k+r$.
By Proposition~\ref{prop:Jacobi-matrices}, reversing the order of summation gives
\[
 \frac12\Ind(\phi_{\lambda,n})
 =\sum_{r=1}^{k}\sum_{\ell=r}^{n-k+r-1}(2\ell+1)
 =\sum_{r=1}^{k}\bigl((n-k+r)^2-r^2\bigr)
 =(n+1)k(n-k).
\]
The zero eigenvalue occurs once in each block with $2\le\ell\le n$, so
\[
 \Nul(\phi_{\lambda,n})
 =2\sum_{\ell=2}^{n}(2\ell+1)
 =2\bigl((n+1)^2-4\bigr)
 =2(n-1)(n+3).
\]
The sums are empty for $n=1$.
\end{proof}

For each $2\le\ell\le n$, fix a nonzero vector $c\in\ker M_\ell$.
Keeping $c$ fixed and varying $j=1,\ldots,2\ell+1$ in the construction of
\eqref{eq:multiplicity-ordering} gives normal Jacobi fields whose complex span
is isomorphic to $V_{2\ell}$, by the right-translation formula in
Section~\ref{subsec:coefficient-modes}.
Since $\dim_\C\ker M_\ell=1$ and no other blocks contribute zero modes,
these copies exhaust the normal Jacobi kernel.

\begin{cor}[Normal Jacobi kernel in $\CP^n$]
\label{cor:linearly-full-kernel-module}
With the complex coefficient convention of
Section~\ref{subsec:normal-connection}, the normal Jacobi kernel has the
$\SU(2)$-module decomposition
\[
 \ker J\cong\bigoplus_{\ell=2}^n V_{2\ell}.
\]
\end{cor}

\begin{remark}
Let $H,A,B$ act on $V_n\otimes V_{2\ell}$ as the sums of their actions
on the two factors.  On the weight-$\lambda$ subspace, the shifted Casimir
operator $\tfrac14(H^2+2AB+2BA-n(n+2)\Id)$ has matrix
$T_\ell$ in the orthonormal basis
\[
 (-1)^d v_{\lambda+2d,n}\otimes v_{-2d,2\ell},\qquad
 -h_\ell\le d\le\min(k,\ell).
\]
Indeed, the ladder identities \eqref{eq:ladder-actions-section2} give
exactly the entries in \eqref{eq:comparison-matrix-entries}, with the
phase $(-1)^d$ producing the negative off-diagonal entries.
\end{remark}

\subsection{Totally geodesic extensions}
\label{subsec:extensions}

\begin{proof}[Proof of Theorem~\ref{thm:introduction-higher-codimension}]
Fix $N\ge n$ and consider the standard totally geodesic inclusion
$\CP^n\subset\CP^N$.  Using the chosen unitary basis, write
\[
 \C^{N+1}=\C^{n+1}\oplus\C^{N-n}
\]
and extend $\rho_n$ trivially on the second factor.  The same projective orbit gives
$\phi_{\lambda,n}^{N}$.  Let $J^N$ be its normal Jacobi operator, so $J^n=J$.

Let $e_{n+2},\ldots,e_{N+1}$ be the standard orthonormal basis of the
complementary factor $\C^{N-n}$.  The lift $Z_\lambda$ and all tangent
representatives lie in $\C^{n+1}$, so these complementary directions are
normal.  With our normalization of the Fubini--Study metric,
$g_{\mathrm{FS}}(e_\alpha,e_\alpha)=4$, so the representatives
$e_\alpha/2$ have unit length.  Relative to $Z_\lambda$, every real normal field along
$\phi_{\lambda,n}^{N}$ has a unique horizontal representative
\[
 V=V_0+\frac12\sum_{\alpha=n+2}^{N+1}f_\alpha e_\alpha,
\]
where $V_0$ represents a normal field along $\phi_{\lambda,n}$ in $\CP^n$
and the $f_\alpha$ are complex-valued functions on $\SU(2)$.
Since $Z_\lambda(t_ug)=e^{\lambda\mathrm{i}u}Z_\lambda(g)$, the horizontal
representative of the same normal vector acquires the same phase.
The vectors $e_\alpha$ are fixed, so
\[
 f_\alpha(t_ug)=e^{\lambda\mathrm{i}u}f_\alpha(g).
\]
Comparing this with the defining equivariance of $L_r$, we obtain
$f_\alpha\in\Gamma(L_{-\lambda/2})$ and the orthogonal splitting
\begin{equation}
\label{eq:higher-normal-splitting}
 N_{\phi_{\lambda,n}^{N}}S^2\simeq N_{\phi_{\lambda,n}}S^2\oplus(L_{-\lambda/2})^{\oplus(N-n)}.
\end{equation}
This is a splitting of real vector bundles, with each complex line viewed
as a real plane.  The original normal summand retains the coefficient
complex structure of Section~\ref{subsec:normal-connection}, while the
additional coefficients use ordinary complex multiplication.

Since the vectors $e_\alpha$ are constant, differentiating an additional
component and subtracting the connection term of $Z_\lambda$ gives, in
horizontal representatives,
\[
 \nabla^\perp\!\left(\frac12f_\alpha e_\alpha\right)
 =\frac12\bigl(\dd f_\alpha-\lambda\mathrm{i}\omega f_\alpha\bigr)e_\alpha.
\]
Thus the additional summands are preserved by the normal connection;
total geodesy also preserves the original normal summand.
On the latter, $J^N$ agrees with $J$.
The vectors $\II(e_i,e_j)$ and $\mathrm{i}e_i$, $i,j=1,2$, lie in $T\CP^n$
and are orthogonal to the additional normal directions.  Hence
$\mathcal B=0$ on these directions, and the Fubini--Study curvature
formula gives $\mathcal R=\tfrac12\Id$.  The second variation formula gives
\begin{equation}
\label{eq:additional-Jacobi-eigenvalue}
 J^{N}=J\oplus
           \left(\Delta+\tfrac12\Id\right)^{\oplus(N-n)}.
\end{equation}
Here $\Delta$ is the connection Laplacian on $L_{-\lambda/2}$ for the induced
metric.  Equation~\eqref{eq:coefficient-modes} decomposes
$L^2(L_{-\lambda/2})$ into summands isomorphic to $V_s$, where
$s\ge|\lambda|$ and $s\equiv n\pmod2$.
The summand corresponding to $s$ has orthogonal basis
$Y_{-\lambda/2,s}^{j}$, $1\le j\le s+1$.
For every $f$ in such a summand, \eqref{eq:line-bundle-spectrum} gives
\begin{equation}
\label{eq:additional-mode-spectrum}
 \left(\Delta+\tfrac12\Id\right)f
 =\frac{(n-s)(n+s+2)}{4D_{\lambda,n}}f,\qquad
 s\ge|\lambda|,\quad s\equiv n\pmod2.
\end{equation}
Here we used $\lambda^2+2D_{\lambda,n}=n(n+2)$.
By completeness, this gives the full spectrum of
$\Delta+\tfrac12\Id$ on sections of $L_{-\lambda/2}$, with complex
multiplicity $s+1$.  The eigenvalues are
positive for $s<n$, zero for $s=n$ and negative for $s>n$.

Each additional summand $L_{-\lambda/2}$ therefore contributes
\[
 2\sum_{\substack{|\lambda|\le s<n\\s\equiv n\ (2)}}(s+1)
 =2\frac{n^2-\lambda^2}{4}=2k(n-k)
\]
to the real index and $2(n+1)$ to the real nullity.
Adding the $N-n$ contributions to the linearly full formulas gives
\[
 \Ind_{\CP^N}(\phi_{\lambda,n}^{N})=2k(n-k)(N+1),\qquad
 \Nul_{\CP^N}(\phi_{\lambda,n}^{N})=2\bigl((N+1)(n+1)-4\bigr).
\]
\end{proof}

On each additional line bundle $L_{-\lambda/2}$,
\eqref{eq:additional-mode-spectrum} gives
\[
 \ker\bigl(\Delta+\tfrac12\Id\bigr)
 =\operatorname{span}_\C\{Y_{-\lambda/2,n}^{j}:1\le j\le n+1\}
 \simeq V_n.
\]
Combining the splitting \eqref{eq:additional-Jacobi-eigenvalue} with
Corollary~\ref{cor:linearly-full-kernel-module} and the preceding kernel identity
gives the following decomposition.

\begin{cor}[Normal Jacobi kernel]
\label{cor:higher-codimension-kernel-module}
With the complex coefficient conventions above, the normal Jacobi
kernel has the $\SU(2)$-module decomposition
\begin{equation}
\label{eq:normal-kernel-module}
 \ker J^{N}
 \simeq\bigoplus_{\ell=2}^nV_{2\ell}\oplus V_n^{\oplus(N-n)}.
\end{equation}
\end{cor}

\subsection{Integrability of normal Jacobi fields}
\label{subsec:integrable-nullity}

We prove Theorem~\ref{thm:introduction-integrable-nullity} by deforming the
holomorphic directrix and comparing the dimension of the resulting space
of normal variational fields with the nullity in
Theorem~\ref{thm:introduction-higher-codimension}.

For a smooth map $f:S^2\to\CP^N$, choose a local nonzero lift $s$ in a
holomorphic coordinate $z$.  Its Gauss transforms are
\cite{LemaireWood2002}
\[
 \partial^{\prime}(f)\coloneqq\left[\partial_zs-
 \frac{\langle\partial_zs,s\rangle_\C}{\langle s,s\rangle_\C}s\right],\qquad
 \partial^{\prime\prime}(f)\coloneqq\left[\partial_{\bar z}s-
 \frac{\langle\partial_{\bar z}s,s\rangle_\C}{\langle s,s\rangle_\C}s\right].
\]
These expressions are independent of the lift and holomorphic coordinate
where the projected derivatives are nonzero.  Let
$P_0:V_n\hookrightarrow\C^{N+1}$ be the standard inclusion,
$P_0(v)=(v,0)$.  For a complex-linear injection $P$ near $P_0$, start
with the holomorphic curve
\[
 \phi_{n,n}^{N}(P):[g]\longmapsto[P(Z_n(g))].
\]
Applying $\partial^{\prime}$ successively defines
\begin{equation}
\label{eq:directrix-deformation-family}
 \phi_{\mu-2,n}^{N}(P)\coloneqq\partial^{\prime}(\phi_{\mu,n}^{N}(P)),
 \qquad \mu=n,n-2,\ldots,2-n.
\end{equation}
Thus, for $\lambda=n-2k$, we have
$\phi_{\lambda,n}^{N}(P)=(\partial^{\prime})^k(\phi_{n,n}^{N}(P))$.

\begin{lem}[Variations of the directrix]
\label{lem:directrix-family}
Let $N\ge n\ge1$, $0\le k\le n$ and $\lambda=n-2k$.
For $P$ near $P_0$, the maps $\phi_{\lambda,n}^{N}(P)$ are minimal
immersions conformal for the fixed complex structure on $S^2$.
The map $(P,x)\mapsto\phi_{\lambda,n}^{N}(P)(x)$ is smooth, and
$\phi_{\lambda,n}^{N}(P_0)=\phi_{\lambda,n}^{N}$.

For $\dot P\in\operatorname{Hom}_\C(V_n,\C^{N+1})$, set
$P_t\coloneqq P_0+t\dot P$ for real $t$ near zero and define
\begin{equation}
\label{eq:matrix-full-velocity}
 U_{\dot P}\coloneqq\left.\frac{\dd}{\dd t}\right|_0
       \phi_{\lambda,n}^{N}(P_t).
\end{equation}
The set
\[
 W\coloneqq\{U_{\dot P}:\dot P\in\operatorname{Hom}_\C(V_n,\C^{N+1})\}
\]
is a real vector space with $\dim_\R W=2(N+1)(n+1)-2$.
\end{lem}

\begin{proof}
For $\mu=n,n-2,\ldots,-n$, equation~\eqref{eq:dZ-mu} gives
\[
 \partial^{\prime}(\phi_{\mu,n}^{N})=\phi_{\mu-2,n}^{N}\quad(\mu>-n),
 \qquad
 \partial^{\prime\prime}(\phi_{\mu,n}^{N})=\phi_{\mu+2,n}^{N}\quad(\mu<n).
\]
The projected derivatives defining these transforms are nowhere zero
in the indicated ranges.  Hence the recursion
\eqref{eq:directrix-deformation-family} at $P_0$ recovers the Veronese
sequence.  Smoothness of the local Gauss formulas, compactness of $S^2$
and openness of the immersion condition give a smooth family of
immersions for $P$ near $P_0$.  These maps form the harmonic sequence
generated by the holomorphic curve $\phi_{n,n}^{N}(P)$
\cite{EellsWood1983,Wolfson1988}.
Each member has vanishing Hopf differential on $S^2$, hence is conformal
for the fixed complex structure and minimal.

The map $\dot P\mapsto U_{\dot P}$ is real-linear by smooth dependence.
To compute $\dim_\R W$, we first prove that
\begin{equation}
\label{eq:full-velocity-kernel}
 U_{\dot P}=0\quad\Longleftrightarrow\quad\dot P\in\C P_0.
\end{equation}
For $c\in\C$, the family $P_t=(1+tc)P_0$ leaves the directrix unchanged,
so $U_{cP_0}=0$.

Conversely, suppose $U_{\dot P}=0$.  Adjacent members of the deformed
harmonic sequence satisfy \cite{BurstallWood1986,ChernWolfson1987}
\begin{equation}
\label{eq:inverse-Gauss-transform}
 \partial^{\prime\prime}(\phi_{\mu,n}^{N}(P_t))=\phi_{\mu+2,n}^{N}(P_t),
 \qquad \mu=\lambda,\lambda+2,\ldots,n-2.
\end{equation}
The projected derivatives defining $\partial^{\prime\prime}$ remain
nonzero near $P_0$, so its local formula depends smoothly on the map
coordinates and their first spatial derivatives.  If a variational field
vanishes identically, so do the spatial derivatives of its coordinate
components.  The chain rule therefore shows that applying
$\partial^{\prime\prime}$ gives a family whose variational field also
vanishes.  Applying \eqref{eq:inverse-Gauss-transform} $k$ times yields
\[
 \left.\frac{\dd}{\dd t}\right|_{t=0}\phi_{n,n}^{N}(P_t)=0.
\]

In the affine coordinate $z=\eta/\zeta$, the directrix has lift $P_t(s(z))$,
where
\[
 s(z)\coloneqq\sum_{q=0}^n\sqrt{\binom nq}\,z^qv_{n-2q,n}.
\]
The directrix velocity vanishes precisely when
\[
 \dot P(s(z))=a(z)P_0(s(z))
\]
for some function $a$.  The first and $(n+1)$-st coordinates give
\[
 a(z)=\bigl(\dot P(s(z))\bigr)_1,\qquad
 z^na(z)=\bigl(\dot P(s(z))\bigr)_{n+1}.
\]
Both right-hand sides are polynomials of degree at most $n$, so $a(z)=c$
is constant.  Comparing coefficients yields $\dot P=cP_0$, proving
\eqref{eq:full-velocity-kernel}.  Since the kernel has real dimension two,
$\dim_\R W=2(N+1)(n+1)-2$.
\end{proof}

\begin{proof}[Proof of Theorem~\ref{thm:introduction-integrable-nullity}]
Let $W$ be the space of variational fields in
Lemma~\ref{lem:directrix-family}.  Every $U\in W$ is the velocity of a
family of conformal minimal immersions, so $U^\perp\in\ker J^N$
\cite{Simons1968}.  Consider the normal projection
\[
 p:W\longrightarrow\ker J^N,\qquad p(U)\coloneqq U^\perp.
\]
By Lemma~\ref{lem:directrix-family} and
Theorem~\ref{thm:introduction-higher-codimension},
$\dim_\R W=\dim_\R\ker J^N+6$, so it suffices to show that
$\dim_\R\ker p\le6$.
For $U\in\ker p$, choose a family $\psi_t$ from the lemma with velocity
$U$, and set $g_0=\psi_0^*g_{\mathrm{FS}}$.
Since $U$ is tangential and $\psi_0$ is an immersion, write
$U=\dd\psi_0(X)$ for a unique vector field $X$ on $S^2$.
Conformality for the fixed complex structure gives
$\psi_t^*g_{\mathrm{FS}}=e^{2u_t}g_0$ with $u_0=0$, and hence
\[
 \mathcal L_Xg_0
 =\left.\frac{\dd}{\dd t}\right|_0\psi_t^*g_{\mathrm{FS}}
 =2\dot u_0g_0.
\]
Thus $X$ is conformal.  The map $U\mapsto X$ is injective, and the space
of conformal vector fields on $S^2$ has real dimension six
\cite[Section~6]{LemaireWood2002}.  Hence $\dim_\R\ker p\le6$.
Consequently,
\begingroup
\begin{equation}
\label{eq:normal-velocity-dimension}
 \dim_\R p(W)=\dim_\R W-\dim_\R\ker p
 \ge\dim_\R W-6=\dim_\R\ker J^N.
\end{equation}
\endgroup
Since $p(W)\subseteq\ker J^N$, the map $p$ is surjective.
For any $V\in\ker J^N$, choose $\dot P$ with $U_{\dot P}^\perp=V$.
Then $\psi_t=\phi_{\lambda,n}^{N}(P_0+t\dot P)$, for sufficiently small
$t$, is the required smooth family of minimal immersions.
\end{proof}

\begin{remark}
The matrix directions in $\operatorname{Hom}_\C(V_n,\C^{N-n})$ move the
original $\CP^n$ inside $\CP^N$; their real dimension $2(N-n)(n+1)$
equals the additional normal nullity.
The integrating families remain minimal but need not be homogeneous or
have constant curvature.
\end{remark}

\appendix
\section{Matrix calculations}
\label{app:matrix-calculations}

We first prove Proposition~\ref{prop:Jacobi-matrices}, then
Lemmas~\ref{lem:comparison-spectrum} and~\ref{lem:matrix-congruence}.

\begin{proof}[Proof of Proposition~\ref{prop:Jacobi-matrices}]
Write $D=D_{\lambda,n}$, $a=a_{\lambda,n}$ and $b=b_{\lambda,n}$.
Thus $D=a^2+b^2$ and $\lambda=a^2-b^2$.
For a normal field $V$ with expansion
\eqref{eq:normal-field-complex-coordinates}, the tangent and normal
frames give, when $0<k<n$,
\[
 \sum_{i=1}^2\langle V,\mathrm{i}e_i\rangle^2
 =\frac{4a^2b^2}{D^2}|A_1|^2.
\]
The sum vanishes when $k=0,n$.  The Fubini--Study curvature formula is
\[
 \langle\mathcal R V,V\rangle
 =\tfrac12|V|^2+\tfrac34\sum_{i=1}^2\langle V,\mathrm{i}e_i\rangle^2.
\]
Together with \eqref{eq:II-frame-formula}, this yields
\begin{equation}
\label{eq:curvature-shape-coefficients}
 D\langle\mathcal B V,V\rangle=\frac{|\alpha A_2^-+\beta A_2^+|^2}{2D},\qquad
 D\langle\mathcal R V,V\rangle=\frac{D}{2}|V|^2+\frac{3a^2b^2}{D}|A_1|^2.
\end{equation}
These quadratic forms show that $\mathcal B,\mathcal R$ are
complex-linear and preserve the line index.
Absent normal coefficients are set to zero in these quadratic forms.

Fix $\ell,j$ with $m_\ell>0$.  We now restrict
\eqref{eq:Jacobi-differential-coefficients} to the corresponding normal
coefficient block.
Let $C_\ell$ be the principal submatrix of $C$ on the coordinates
of line index at most $\ell$, and let $Z_\ell$ be the corresponding
principal submatrix of
\[
 -\frac{C^{\mathrm T}C+CC^{\mathrm T}}{2D}+\mathcal B+\mathcal R.
\]
Both products are computed on the full normal coefficient space before
restriction, so a coupling from $r=\ell$ to $r=\ell+1$ still contributes
to the diagonal.  In the coefficient order of
Section~\ref{subsec:normal-connection}, write a vector in this block as
$x=(x_1,(x_r^-)_{r=2}^{h_\ell},(x_r^+)_{r=2}^{\min(k,\ell)})^{\mathrm T}$.
The central coordinate $x_1$ is omitted when $k=0,n$.
In the formulas below, every coordinate absent from this block is
set to zero.  We have
\[
 D\,x^*Z_\ell x=(D+1)|x_1|^2
 +\sum_{r=2}^{h_\ell}r(r-\lambda)|x_r^-|^2
 +\sum_{r=2}^{\min(k,\ell)}r(r+\lambda)|x_r^+|^2
 +\frac{|\alpha x_2^-+\beta x_2^+|^2}{D}.
\]
To verify this identity, first suppose $0<k<n$.  In the full coefficient
order $A_1,A_2^-,\ldots,A_{n-k}^-,A_2^+,\ldots,A_k^+$,
\eqref{eq:normal-connection-components} gives
\[
 \begin{aligned}
 (C^{\mathrm T}C)_{ij}&=0 &&(i\ne j),\\[3pt]
 (C^{\mathrm T}C+CC^{\mathrm T})_{r,r}
 &=D+2r\lambda-2r^2-\frac{\delta_{r2}\alpha^2}{D}
 &&(2\le r\le n-k),\\[3pt]
 (C^{\mathrm T}C+CC^{\mathrm T})_{n-k+r-1,n-k+r-1}
 &=D-2r\lambda-2r^2-\frac{\delta_{r2}\beta^2}{D}
 &&(2\le r\le k).
 \end{aligned}
\]
The $r=2$ corrections come from the $A_1$-terms in
\eqref{eq:normal-connection-components}.
\begingroup
For the lower branch, when $n-k\ge2$,
\[
 (C^{\mathrm T}C+CC^{\mathrm T})_{2,2}
 =\frac{b^2a_{\lambda-2,n}^2}{D}+a_{\lambda-4,n}^2
 =D+4\lambda-8-\frac{\alpha^2}{D}.
\]
\endgroup
For the central diagonal entry,
\[
 (C^{\mathrm T}C+CC^{\mathrm T})_{1,1}
 =\frac{b^2a_{\lambda-2,n}^2+a^2b_{\lambda+2,n}^2}{D}
 =\frac{6a^2b^2-D^2-2D}{D}.
\]
Up to symmetry, the only off-diagonal entry of $CC^{\mathrm T}$ is
\[
 (CC^{\mathrm T})_{2,n-k+1}=-\frac{\alpha\beta}{D},
\]
when both branches contain an $r=2$ coordinate.  If both occur in the
$\ell$-block, \eqref{eq:curvature-shape-coefficients} gives
\[
 (Z_\ell)_{2,h_\ell+1}
 =-\frac{-\alpha\beta/D}{2D}+\frac{\alpha\beta}{2D^2}
 =\frac{\alpha\beta}{D^2}.
\]
Together with \eqref{eq:curvature-shape-coefficients}, these entries give
the quadratic form $D\,x^*Z_\ell x$ above.  At $k=0$ or
$k=n$, the incoming coupling at $r=2$ is absent; the corresponding ladder
coefficient satisfies $a_{n-2,n}^2=2(n-1)$ or $b_{-n+2,n}^2=2(n-1)$,
respectively.  Thus the same identity holds with $D=n$ and, respectively,
$\alpha^2=2n(n-1)$ or $\beta^2=2n(n-1)$.

Let $r_p$ be the line index of coordinate $p$ in this block, and put
\[
 B_\ell\coloneqq\operatorname{diag}(b_{2r_p,2\ell})_{p=1}^{m_\ell}.
\]
The basis $E_1,\ldots,E_{m_\ell}$ differs from the coefficient basis
with entries $Y_{r_p,2\ell}^j$ only by the diagonal sign matrix $S_\ell$,
whose entries are $(-1)^{r-1}$ on the lower branch and $1$ elsewhere.
If $(C_\ell)_{pq}\ne0$, then $r_p=r_q+1$ and
$a_{2r_p,2\ell}=b_{2r_q,2\ell}$.
Consequently, in the coefficient basis before applying $S_\ell$,
the two first-order terms in
\eqref{eq:Jacobi-differential-coefficients} are represented by
$-C_\ell B_\ell/D$ and $-B_\ell C_\ell^{\mathrm T}/D$.
Using \eqref{eq:line-bundle-spectrum} and the ladder identities
\eqref{eq:coefficient-ladder} with the coefficients
\eqref{eq:normal-mode-coefficients}, we obtain
\[
 M_\ell=S_\ell\left[
 \frac{\operatorname{diag}(r_p^2-\ell(\ell+1))_{p=1}^{m_\ell}
       -C_\ell B_\ell-B_\ell C_\ell^{\mathrm T}}{D}
 +Z_\ell\right]S_\ell.
\]
The diagonal and adjacent entries give
\eqref{eq:Jacobi-matrix-diagonals}--\eqref{eq:Jacobi-matrix-edges}.
When both $r=2$ coordinates occur, the remaining coupling is
\[
 (M_\ell)_{2,h_\ell+1}=-(Z_\ell)_{2,h_\ell+1}
 =-\frac{\alpha\beta}{D^2}.
\]
All other off-diagonal entries vanish.
The preceding formula for the full matrix $M_\ell$ also gives
the endpoint cases, with correction
$2(n-1)/n$ at $r=2$.  For $\ell=1$, it gives $M_1=(1)$ when
$0<k<n$, and there is no normal block when $k=0,n$.

The matrix in brackets is real symmetric, and hence so is $M_\ell$.
The spectral decomposition and multiplicities follow from
Section~\ref{subsec:coefficient-modes}.
\end{proof}

\begin{proof}[Proof of Lemma~\ref{lem:comparison-spectrum}]
Write $h=h_\ell$.  For $1\le i\le m_\ell+2$, put $d=i-h-1$ and define
\[
 c_i\coloneqq(-1)^d\sqrt{\binom{2\ell}{\ell+d}\binom n{k-d}},\qquad
 (u_j)_i\coloneqq c_i\binom{i-1}{j}\quad(0\le j\le m_\ell+1).
\]
The matrix with columns $u_0,\ldots,u_{m_\ell+1}$ is lower triangular
with nonzero diagonal entries $c_1,\ldots,c_{m_\ell+2}$, so these
vectors form a basis.  For adjacent rows, the binomial coefficient
ratios give
\[
 (T_\ell)_{i,i+1}\frac{c_{i+1}}{c_i}=(\ell-d)(k-d),\qquad
 (T_\ell)_{i,i-1}\frac{c_{i-1}}{c_i}=(\ell+d)(n-k+d).
\]
The sum of these two coefficients is $\ell n+\lambda d+2d^2$;
adding the diagonal entry of $T_\ell$ gives $\ell(n+\ell+1)$.
Set $x=i-1=d+h$.  Since $h=\min(n-k,\ell)$,
\[
 (\ell+d)(n-k+d)=x(x+\ell+n-k-2h).
\]
We regard $\binom{x}{j}$ as a polynomial in $x$ and set
$\binom{x}{-1}=0$.  The identities
\[
 \binom{x+1}{j}-\binom{x}{j}=\binom{x}{j-1},\qquad
 x\left(\binom{x-1}{j}-\binom{x}{j}\right)=-j\binom{x}{j}
\]
therefore give
\begin{equation}
\label{eq:comparison-difference-operator}
\frac{(T_\ell u_j)_i}{c_i}
=\ell(n+\ell+1)\binom{x}{j}
+(\ell+h-x)(k+h-x)\binom{x}{j-1}
-j(x+\ell+n-k-2h)\binom{x}{j}.
\end{equation}
At the first and last rows, the coefficients of the missing backward
and forward terms vanish, respectively, so the same formula applies.
For $j\ge1$, the identity
$(x-j+1)\binom{x}{j-1}=j\binom{x}{j}$ yields
\[
(\ell+h-x)(k+h-x)\binom{x}{j-1}
=j(x+j-1-\ell-k-2h)\binom{x}{j}
+(\ell+h-j+1)(k+h-j+1)\binom{x}{j-1}.
\]
Substitution into \eqref{eq:comparison-difference-operator} gives
\[
 T_\ell u_j=(\ell-j)(n+\ell-j+1)u_j
 +(\ell+h-j+1)(k+h-j+1)u_{j-1},\qquad 0\le j\le m_\ell+1,
\]
where $u_{-1}=0$; the case $j=0$ follows directly from
\eqref{eq:comparison-difference-operator}.
Thus $T_\ell$ is similar to an upper bidiagonal matrix, whose diagonal entries are
$(\ell-j)(n+\ell-j+1)$.  Successive entries differ by
$n+2\ell-2j>0$ for $0\le j\le m_\ell$, since
$m_\ell+1\le\min(n,2\ell)$.  This proves the stated simple spectrum.
\end{proof}

\begin{proof}[Proof of Lemma~\ref{lem:matrix-congruence}]
Write $t=\ell(\ell+1)$, $D=D_{\lambda,n}$, $a=a_{\lambda,n}$ and
$b=b_{\lambda,n}$, so $D=a^2+b^2$ and $\lambda=a^2-b^2$.
For $x\in\R^{m_\ell+2}$, index its coordinates by
$-h_\ell\le d\le\min(k,\ell)$ and set $x_d=0$ outside this range.
Put
\[
 u\coloneqq x_0-\frac{ax_{-1}+bx_1}{\sqrt t},\qquad
 v\coloneqq\frac{ax_{-1}-bx_1}{\sqrt D},\qquad
 z\coloneqq\frac{bx_{-1}+ax_1}{\sqrt D}.
\]
For $0<k<n$, the change from $(x_{-1},x_1)$ to $(v,z)$ is orthogonal.
At $k=0$ or $k=n$, $z=0$ and $v=x_{-1}$ or $v=-x_1$, respectively.
\begingroup
The inverse relations are
\[
 x_{-1}=\frac{av+bz}{\sqrt D},\qquad
 x_1=\frac{-bv+az}{\sqrt D}.
\]
\endgroup
The terms of $x^{\mathrm T}T_\ell x$ involving only $x_0,x_{-1},x_1$
satisfy
\begin{equation}
\label{eq:congruence-central-form}
 \begin{aligned}
 &t x_0^2-2\sqrt t\,x_0(ax_{-1}+bx_1)
   +(t+\lambda-2)x_{-1}^2+(t-\lambda-2)x_1^2\\
 &\qquad=t u^2+(t-2)(x_{-1}^2+x_1^2)-(bx_{-1}+ax_1)^2=t u^2+(t-2)v^2+(t-D-2)z^2.
 \end{aligned}
\end{equation}
For $\ell=1$, there are no other coordinates.  This identity gives
$\operatorname{diag}(2,0,-D)$ when $0<k<n$, since $M_1=(1)$, and
$\operatorname{diag}(2,0)$ when $k=0,n$.

Now suppose $\ell\ge2$.  The only terms coupling $x_{-1},x_1$ to
the remaining coordinates are
\begin{equation}
\label{eq:congruence-adjacent-form}
-2\sqrt{t-2}\bigl(a_{\lambda-2,n}x_{-1}x_{-2}
+b_{\lambda+2,n}x_1x_2\bigr)
=-2\sqrt{\frac{t-2}{D}}\left[
v(\alpha x_{-2}-\beta x_2)
+z\bigl(ba_{\lambda-2,n}x_{-2}+ab_{\lambda+2,n}x_2\bigr)\right].
\end{equation}
Define
\[
 \widehat v\coloneqq v-\frac{\alpha x_{-2}-\beta x_2}{\sqrt{D(t-2)}},\qquad
 y\coloneqq\bigl(z,(x_{-r})_{r=2}^{h_\ell},
                   (x_r)_{r=2}^{\min(k,\ell)}\bigr)^{\mathrm T},
\]
omitting $z$ when $k=0,n$, so that $y\in\R^{m_\ell}$.
Completing the square in $v$ leaves the correction
\[
 -\frac{(\alpha x_{-2}-\beta x_2)^2}{D}
 =-\frac{\alpha^2}{D}x_{-2}^2
   +\frac{2\alpha\beta}{D}x_{-2}x_2
   -\frac{\beta^2}{D}x_2^2.
\]
This gives the two $r=2$ diagonal corrections and the cross-branch
entry of $-DM_\ell$.  The diagonal entry corresponding to $z$
and its couplings to $x_{-2},x_2$ are given by \eqref{eq:congruence-central-form} and
\eqref{eq:congruence-adjacent-form}.
On the lower and upper branches, the unchanged diagonal entries of
$T_\ell$ are $t-r(2r-\lambda)$ and $t-r(2r+\lambda)$, respectively.
Its unchanged within-branch entries are also precisely those of
$-DM_\ell$, by \eqref{eq:comparison-matrix-entries} and
\eqref{eq:Jacobi-matrix-edges}--\eqref{eq:Jacobi-matrix-endpoints}.
All other entries vanish.  Comparing with
\eqref{eq:Jacobi-matrix-diagonals}--\eqref{eq:Jacobi-matrix-endpoints} gives
\begin{equation}
\label{eq:congruence-square-completion}
 x^{\mathrm T}T_\ell x=t u^2+(t-2)\widehat v^{\,2}
                         -D\,y^{\mathrm T}M_\ell y.
\end{equation}
For $k=0,n$, the same calculation gives the stated single-branch
matrix: $D=n$ and the $r=2$ correction is
$-\alpha^2/n=-2(n-1)$ or $-\beta^2/n=-2(n-1)$, respectively.

The change of coordinates from $x$ to $(u,\widehat v,y)$ is invertible:
the change from $(x_{-1},x_1)$ to $(v,z)$ is orthogonal
(a sign change when $k=0,n$), and the subtractions defining $u$
and $\widehat v$ are triangular changes.
If $P$ is the matrix of the inverse change of coordinates, then
\eqref{eq:congruence-square-completion} gives
\eqref{eq:matrix-congruence}.
For $\ell=1$, use $(u,v,z)$, omitting $z$ when $k=0,n$.
\end{proof}

\begingroup
\makeatletter
\@startsection{section}{1}{\z@}%
  {18\p@ \@plus 6\p@ \@minus 3\p@}{1sp}%
  {\normalsize\bfseries\boldmath}*{Funding}
\makeatother
\setlength{\emergencystretch}{1.5em}
\noindent This work was supported by the National Natural Science Foundation of China
[grant numbers 12301068, 12571057]; the Chinese Academy of Sciences
[grant number YSBR-001]; the Xiaomi Foundation through the Xiaomi Young
Scholar Program; and the Fundamental Research Funds for the Central
Universities [grant number WK0010000081].
\par
\endgroup

\section*{Acknowledgments}

The authors thank Prof. Qing Chen, Prof. Xiaoxiang Jiao and Prof. Xiaowei Xu.

\Needspace{9\baselineskip}
\section*{Declarations}
\begingroup
\setlength{\parskip}{0pt}
\noindent\textit{Conflict of Interest.} The authors declare that they have no conflicts of interest.\par
\noindent\textit{Author Contributions.} Both authors contributed equally and are listed alphabetically.\par
\noindent\textit{Data Availability.} No datasets were generated or analyzed for this theoretical study.\par
\noindent\textit{Use of AI.} The authors used ChatGPT (OpenAI) to improve the English language and readability of this manuscript. The authors take full responsibility for its content.\par
\endgroup


\begingroup
\makeatletter
\@secpenalty=0\relax
\makeatother
\setlength{\bibsep}{0pt}

\endgroup

\end{document}